\documentclass[10pt,reqno]{amsart}
\usepackage{amsmath,amscd,amssymb }
\usepackage{enumerate}

\newtheorem{proposition}{Proposition}[section]
\newtheorem{lemma}[proposition]{Lemma}
\newtheorem{theorem}[proposition]{Theorem}
\newtheorem{corollary}[proposition]{Corollary}
\newtheorem{conjecture}{Conjecture}[section]

\theoremstyle{definition}
\newtheorem{remark}[proposition]{Remark}
\newtheorem{definition}[proposition]{Definition}

\newcommand{\R}{\mathbb{R}}
\newcommand{\C}{\mathbb{C}}
\newcommand{\A}{\mathbb{A}}
\newcommand{\Z}{\mathbb{Z}}
\newcommand{\Q}{\mathbb{Q}}

\newcommand{\pr}{\mathbb{P}}

\newcommand{\scK}{\mathcal{K}}
\newcommand{\scL}{\mathcal{L}}

\newcommand{\scI}{\mathcal{I}}
\newcommand{\scD}{\mathcal{D}}

\newcommand{\scX}{\mathcal{X}}
\newcommand{\scO}{\mathcal{O}}
\newcommand{\scR}{\mathcal{R}}
\newcommand{\scB}{\mathcal{B}}
\newcommand{\scY}{\mathcal{Y}}

\newcommand{\scP}{\mathcal{P}}

\DeclareMathOperator{\lct}{lct}

\DeclareMathOperator{\Aut}{Aut}

\DeclareMathOperator{\Supp}{Supp}

\DeclareMathOperator{\DF}{DF}

\DeclareMathOperator{\mult}{mult}

\DeclareMathOperator{\val}{val}

\title[Alpha invariants and K-stability]{Alpha invariants and K-stability for general polarisations of Fano varieties}

\author{Ruadha\'i Dervan}

\begin{document}

\maketitle

\begin{abstract}
We provide a sufficient condition for polarisations of Fano varieties to be K-stable in terms of Tian's alpha invariant, which uses the log canonical threshold to measure singularities of divisors in the linear system associated to the polarisation. This generalises a result of Odaka-Sano in the anti-canonically polarised case, which is the algebraic counterpart of Tian's analytic criterion implying the existence of a K\"{a}hler-Einstein metric. As an application, we give new K-stable polarisations of a general degree one del Pezzo surface. We also prove a corresponding result for log K-stability.
\end{abstract}

\tableofcontents

\section{Introduction}

A central problem in complex geometry is to find necessary and sufficient conditions for the existence of a constant scalar curvature K\"{a}hler (cscK) metric in a given K\"{a}hler class. One of the first sufficient conditions is due to Tian, who introduced the alpha invariant. The alpha invariant $\alpha(X,L)$ of a polarised variety $(X,L)$ is defined as the infimum of the log canonical thresholds of $\Q$-divisors in the linear system associated to $L$, measuring singularities of these divisors. Tian \cite{GT} proved that if $X$ is a Fano variety of dimension $n$ with canonical divisor $K_X$, the lower bound $\alpha(X,-K_X)>\frac{n}{n+1}$ implies that $X$ admits a K\"{a}hler-Einstein metric in $c_1(X)=c_1(-K_X)$. 

The Yau-Tian-Donaldson conjecture states that the existence of a cscK metric in $c_1(L)$ for a polarised manifold $(X,L)$ is equivalent to the algebro-geometric notion of K-stability, related to geometric invariant theory. This conjecture has recently been proven in the case that $L=-K_X$ \cite{CDS2, CDS3,CDS4,GT2}. By work of Donaldson \cite{SD} and Stoppa \cite{JS1}, it is known that the existence of a cscK metric in $c_1(L)$ implies that $(X,L)$ is K-stable, provided the automorphism group of $X$ is discrete. Odaka-Sano \cite{OS} have given a direct algebraic proof that $\alpha(X,-K_X)>\frac{n}{n+1}$ implies that $(X,-K_X)$ is K-stable. This provides the first algebraic proof of K-stability of varieties of dimension greater than one. 

On the other hand, few sufficient criteria are known for K-stability in the general case. We give a sufficient condition for general polarisations of Fano varieties to be K-stable. A fundamental quantity will be the slope of a polarised variety $(X,L)$, defined as \begin{equation}\mu(X,L)=\frac{-K_X.L^{n-1}}{L^n}=\frac{\int_Xc_1(X).c_1(L)^{n-1}}{\int_Xc_1(L)^n}.\end{equation} The slope is therefore a topological quantity which, after rescaling $L$, can be assumed equal to $1$. Our main result is then as follows. 

\begin{theorem}\label{intromaintheorem} Let $(X,L)$ be a polarised $\Q$-Gorenstein log canonical variety with canonical divisor $K_X$. Suppose that 
\begin{itemize}
\item[(i)]  $\alpha(X,L)>\frac{n}{n+1}\mu(X,L)$ and
\item[(ii)]  $-K_X \geq \frac{n}{n+1}\mu(X,L) L$.
\end{itemize}
Then $(X,L)$ is K-stable. \end{theorem}

Here, for divisors $H,H'$, we write $H\geq H'$ to mean $H-H'$ is nef. Note that when $L=-K_X$, the slope of $(X,L)$ is equal to $1$, the second condition is vacuous and this theorem is then due to Odaka-Sano. The condition that $X$ is log canonical ensures that $\alpha(X,L)\geq 0$, while the condition that $X$ is $\Q$-Gorenstein ensures that $-K_X$ exists as a $\Q$-Cartier divisor. The second condition implies that $X$ is either Fano or numerically Calabi-Yau, see Remark \ref{calabiyau}. By proving a continuity result for the alpha invariant, we also show in Corollary \ref{open} that provided the inequality in the second condition is strict, the conditions to apply Theorem $\ref{intromaintheorem}$ are \emph{open} when varying the polarisation. 

Theorem \ref{intromaintheorem} gives the first non-toric criterion for K-stability of general polarisations of Fano varieties. On the analytic side, a result of LeBrun-Simanca \cite{LS} states that the condition that a polarised variety $(X,L)$ admits a cscK metric is an \emph{open} condition when varying $L$, provided the automorphism group of $X$ is discrete. As the existence of a cscK metric in $c_1(L)$ implies K-stability, this gives an analytic proof that in the situation of \ref{intromaintheorem}, K-stability is an open condition again in the case that the automorphism group of $X$ is discrete. On the other hand, our result can also be used to give \emph{explicit} K-stable polarisations, see for example Theorem \ref{introex}. 

Many computations \cite{CDS, IC} of alpha invariants have been done for anti-canonically polarised Fano varieties. Cheltsov \cite{IC}, building on work of Park \cite{JP}, has calculated alpha invariants of del Pezzo surfaces. As a corollary, Cheltsov's results imply that general anti-canonically polarised del Pezzo surfaces of degrees one, two and three are K-stable. Following the method of proof of Cheltsov, we give new examples of K-stable polarisations of a general del Pezzo surface $X$ of degree one. Noting that $X$ is isomorphic to a blow-up of $\pr^2$ at $8$ points in general position, we denote by $H$ the hyperplane divisor, $E_i$ the $8$ exceptional divisors and $L_{\lambda} = 3H - \sum_{i=1}^7E_i - \lambda E_8$ arising from this isomorphism.

\begin{theorem}\label{introex} $(X,L_{\lambda})$ is K-stable for \begin{equation}\frac{19}{25}\approx \frac{1}{9}(10-\sqrt{10}) < \lambda < \sqrt{10}-2\approx \frac{29}{25}.\end{equation} \end{theorem}

We note that Theorem \ref{intromaintheorem} is merely \emph{sufficient} to prove K-stability. It would be interesting to know exactly which polarisations of a general degree one del Pezzo surface are K-stable. Analytically, a result of Arezzo-Pacard \cite{AP} implies that $(X,L_{\lambda})$ admits a cscK metric for $\lambda$ sufficiently small. In particular, by work of Donaldson \cite{SD} and Stoppa \cite{JS1}, this implies $(X,L_{\lambda})$ is K-stable for $\lambda$ sufficiently small. However, using the technique of slope stability, Ross-Thomas \cite[Example 5.30]{RT} have shown that there are polarisations of such an $X$ which are K-\emph{un}stable. 

The recent proof of the Yau-Tian-Donaldson conjecture \cite{CDS2, CDS3,CDS4, GT2} in the case $L=-K_X$ has emphasised the importance of log K-stability. This concept extends K-stability to pairs $(X,D)$ and conjecturally corresponds to cscK metrics with cone singularities along $D$. With this in mind, we extend Theorem \ref{intromaintheorem} to the log setting as follows.

\begin{theorem}\label{intrologmaintheorem} Let $((X,D);L)$ consist of a $\Q$-Gorenstein log canonical pair $(X,D)$ with canonical divisor $K_X$, such that $D$ is an effective integral reduced Cartier divisor on a polarised variety $(X,L)$. Denote $\mu_{\beta}((X,D);L)= \frac{-(K_X+(1-\beta)D).L^{n-1}}{L^n}$. Suppose that 
\begin{itemize}
\item[(i)]  $\alpha((X,D);L)>\frac{n}{n+1}\mu_{\beta}((X,D);L)$ and
\item[(ii)]  $-(K_X +(1-\beta)D)\geq \frac{n}{n+1}\mu_{\beta}((X,D);L) L$.
\end{itemize}
Then $((X,D);L)$ is log K-stable with cone angle $\beta$ along $D$.\end{theorem}

\noindent {\bf Notation and conventions:} By a polarised variety $(X,L)$ we mean a normal complex projective variety $X$ together with an ample line bundle $L$. We often use the same letter to denote a divisor and the associated line bundle, and mix multiplicative and additive notation for line bundles.

\section{Prerequisites}
\subsection{K-stability}

K-stability of a polarised variety $(X,L)$ is an algebraic notion conjecturally equivalent to the existence of a constant scalar curvature K\"{a}hler metric in $c_1(L)$, which requires the so-called Donaldson-Futaki invariant to be positive for all non-trivial test configurations.

\begin{definition} A \emph{test configuration} for a normal polarised variety $(X,L)$ is a normal polarised variety $(\scX,\scL)$ together with
\begin{itemize} 
\item a proper flat morphism $\pi: \scX \to \C$,
\item a $\C^*$-action on $\scX$ covering the natural action on $\C$,
\item and an equivariant very ample line bundle $\scL$ on $\scX$
\end{itemize}
such that the fibre $(\scX_t,\scL_t)$ over $t$ is isomorphic to $(X,L)$ for one, and hence all, $t \in \C^*$. 
\end{definition}

\begin{definition}\label{triviality} We say that a test configuration is \emph{almost trivial} if $X$ is $\C^*$-isomorphic to the product configuration away from a closed subscheme of codimension at least 2. \end{definition} 

\begin{definition} We will later be interested in a slightly modified version of test configurations. In particular, we will be interested in the case where we have a proper flat morphism $\pi: \scX \to \pr^1$ with target $\pr^1$ rather than $\C$ such that $\scL$ is just \emph{relatively} semi-ample over $\pr^1$, that is, a multiple of the restriction to each fibre over $\pr^1$ is basepoint free. We call this a \emph{semi-test configuration}. \end{definition}

As the $\C^*$-action on $(\scX,\scL)$ fixes  the central fibre $(\scX_0,\scL_0)$, there is an induced action on $H^0(\scX_0,\scL^k_0)$ for all $k$. Denote by $w(k)$ the total weight of this action, which is a polynomial in $k$ of degree $n+1$ for $k\gg 0$, where $n$ is the dimension of $X$. Denote the Hilbert polynomial of $(X,L)$ as \begin{equation}\scP(k) = \chi(X,L^k)=a_0k^n+a_1k^{n-1}+O(k^{n-2})\end{equation}and denote also the total weight of the $\C^*$-action on $H^0(\scX_0,\scL^k_0)$ as \begin{equation}w(k) = b_0k^{n+1}+b_1k^n+O(k^{n-1}).\end{equation}

\begin{definition}We define the \emph{Donaldson-Futaki invariant} of a test configuration $(\scX,\scL)$ to be \begin{equation}\DF(\scX,\scL)=\frac{b_0a_1-b_1a_0}{a_0^2}.\end{equation} We say $(X,L)$ is \emph{K-stable} if $\DF(\scX,\scL)>0$ for all test configurations which are not almost trivial. \end{definition}

\begin{remark} For more information on the following remarks, or for a more detailed discussion of K-stability, see \cite{RT}. 
\begin{itemize}
\item The definition of K-stability is independent of scaling $L\to L^r$. In particular, it makes sense for pairs $(X,L)$ where $X$ is a variety and $L$ is a $\Q$-line bundle.
\item If one expands $\frac{w(k)}{k\scP(k)} = f_0 + f_1k^{-1} + O(k^{-2})$, the Donaldson-Futaki invariant is given by $f_1$.
\item One should think of test configuration as geometrisations of the one-parameter subgroups that are considered when applying the Hilbert-Mumford criterion to GIT stability. In fact, asymptotic Hilbert stability implies K-semistability, since the Donaldson-Futaki invariant appears as the leading coefficient in a polynomial associated with asymptotic Hilbert stability.
\item The notion of almost trivial test configurations was introduced by Stoppa \cite{JS3} to resolve a pathology noted by Li-Xu \cite[Section 2.2]{LX}. 
\end{itemize}\end{remark}

\begin{conjecture}(Yau-Tian-Donaldson) A smooth polarised variety $(X,L)$ admits a constant scalar curvature metric in $c_1(L)$ if and only if $(X,L)$ is K-stable. \end{conjecture}

\begin{remark} This conjecture as stated has recently been proven by Chen-Donaldson-Sun \cite{CDS2, CDS3, CDS4} and separately Tian \cite{GT2} in the case $L=-K_X$ (so $X$ is \emph{Fano}). It is expected to hold in the general case, with possibly some slight modifications to the definition of K-stability, see \cite{GS}.\end{remark}

\subsection{Odaka's Blowing-up Formalism}
In \cite{O1}, Odaka shows that to check K-stability, it suffices to check the positivity of the Donaldson-Futaki invariant on \emph{semi}-test configurations arising from flag ideals.

\begin{definition} A flag ideal on $X$ is a coherent ideal sheaf $\scI$ on $X\times \A^1$ of the form $\scI = I_0 + (t)I_1 +\hdots +(t^N)$ with $I_0\subseteq I_1 \subseteq \hdots \subseteq I_{N-1}\subseteq \scO_X$ a sequence of coherent ideal sheaves. The ideal sheaves $I_j$ thus correspond to subschemes $Z_0\supseteq Z_1\supseteq\hdots\supseteq Z_{N-1}$ of $X$. Flag ideals can be equivalently characterised by being $\C^*$-invariant with support on $X\times\{0\}$. \end{definition}

\begin{remark}\label{flagconventions} The flag ideal $\scI$ naturally induces a coherent ideal sheaf on $X\times\pr^1$, which we also denote by $\scI$. Blowing-up $\scI$ on $X\times\pr^1$, we get a map \begin{equation} \pi: \scB=Bl_{\scI}(X\times\pr^1)\to X\times\pr^1.\end{equation} Denote by $E$ the exceptional divisor of the blow-up $\pi: \scB\to X\times\pr^1$, that is, $\scO(-E)=\pi^{-1}\scI$. Abusing notation, write $\scL-E$ to denote $(p_1\circ \pi)^*L\otimes \scO(-E)$, where $p_1:X\times\pr^1\to X$ is the natural projection. Note that the induced map from $\scB\to \pr^1$ is flat by \cite[Remark 5.2]{RT}. There is a natural $\C^*$ action on $X\times\pr^1$, acting trivially on $X$, which lifts to an action on $\scB$. With this action, provided $\scL-E$ is relatively semi-ample over $\pr^1$ and $\scB$ is normal, we have that $(\scB,\scL-E)$ is a \emph{semi}-test configuration.\end{remark}

\begin{theorem}\cite[Corollary 3.11]{O1}  Assume that $(X,L)$ is a normal polarised variety. Then $(X,L)$ is K-stable if and only if $\DF(\scB,\scL^r-E)>0$ for all $r>0$ and for all flag ideals $\scI\neq(t^N)$ with $\scB$ normal and Gorenstein in codimension one and with $\scL^r-E$ relatively semi-ample over $\pr^1$. \end{theorem}

\begin{remark} That $\scB$ can be assumed normal was noted by Odaka-Sano \cite[Proposition 2.1]{OS}. The condition that $\scI\neq(t^N)$ is to ensure $\scB$ is not almost trivial, see Definition \ref{triviality}.\end{remark}

\begin{remark} As a general test configuration $(\scX,\scL)$ is $\C^*$-isomorphic to $(X\times\A^1,L^r)$ away from the central fibre, it is $\C^*$-birational to $(X\times\A^1,L)$. In particular, it is dominated by a blow-up of $X\times\A^1$ along a flag ideal. Odaka shows that one can choose a flag ideal such that the Donaldson-Futaki invariant of the two test configurations are equal. In order to use the machinery of intersection theory, one must also compactify $X\times\A^1$ to $X\times\pr^1$. \end{remark}

\begin{remark} In the case the flag is of the form $\scI = I_0 + (t)$, blowing-up $\scI$ on $X\times \A^1$ leads to \emph{deformation to the normal cone}. In \cite{RT2}, Ross-Thomas study test configurations arising from this process. Stability with respect to test configurations arising from blow-ups of the form $\scI = I_0 + (t)$ is called \emph{slope stability}. Note that Panov-Ross \cite[Example 7.8]{PR} have shown that the blow-up of $\pr^2$ at $2$ points is slope stable but is not K-stable. One must therefore consider more general flag ideals to check K-stability. \end{remark} 

One benefit of this formalism is that, for test configurations arising from flag ideals, there is an explicit intersection-theoretic formula for the Donaldson-Futaki invariant.

\begin{theorem}\label{dfformulapf}\cite[Theorem 3.2]{O1} For a semi-test configuration of the form $(\scB = Bl_{\scI}X\times\pr^1, \scL-E)$ arising from a flag ideal $\scI$ with $\scB$ normal and Gorenstein in codimension one, the Donaldson-Futaki invariant is given by (up to multiplication by a positive constant) \begin{equation}\label{dfform} \DF = -n(L^{n-1}.K_X)(\scL-E)^{n+1} + (n+1)(L^n)(\scL-E)^n.(\scK_X + K_{\scB/X\times\pr^1}).\end{equation}Here we have denoted by $\scK_X$ the pull back of $K_X$ to $\scB$. The intersection numbers $L^{n-1}.K_X$ and $L^n$ are computed on $X$, while the remaining intersection numbers are computed on $\scB$. Replacing $L$ and $\scL$ by $L^r$ and $\scL^r$ respectively in formula \ref{dfform} gives the formula for the Donaldson-Futaki invariant of a test configuration of the form $(\scB, \scL^r-E)$.
\end{theorem}

Note that $\scK_X + K_{\scB/X\times\pr^1} = K_{\scB/\pr^1}$. The benefit of splitting this into two terms is that positivity of the contribution from the second term, the relative canonical divisor over $X\times\pr^1$, can be controlled under assumptions on the singularities of $X$.

\subsection{Log Canonical Thresholds}
The log canonical threshold of a pair $(X,D)$ is a measure of singularity, related to the complex singularity exponent. It takes into consideration both the singularities of $X$ and $D$. See \cite{JK} for more information on log canonical thresholds.

\begin{definition}\label{lc} Let $X$ be a normal variety and let $D=\sum d_iD_i$ be a divisor on $X$ such that $K_X+D$ is $\Q$-Cartier, where $D_i$ are prime divisors. Let $\pi: Y\to X$ be a log resolution of singularities, so that $Y$ is smooth and $\pi^{-1}D \cup E$ has simple normal crossing support where $E$ is the exceptional divisor. We can then write \begin{equation}K_Y - \pi^*(K_X+D) \equiv \sum a(E_i,(X,D))E_i\end{equation} where $E_i$ is either an exceptional divisor or $E_i=\pi_*^{-1}D_i$ for some $i$. That is, either $E_i$ is exceptional or the proper transform of a component of $D$. We usually abbreviate $a(E_i,(X,D))$ to $a_i$. We say the pair $(X,D)$ is \begin{itemize}
\item \emph{log canonical} if $a(E_i,(X,D))\geq -1$ for all $E_i$,
\item \emph{Kawamata log terminal} if $a(E_i,(X,D)) > -1$ for all $E_i$.
\end{itemize} By \cite[Lemma 3.10]{JK} these notions are independent of log resolution.\end{definition}

We will later need a form of inversion of adjunction for log canonicity.

\begin{theorem}\cite{MK}\label{inversionofadjunction} Let $D=D'+D''$ be a $\Q$-divisor with $D'$ an effective reduced normal Cartier divisor and $D''$ an effective $\Q$-divisor which has no common components with $D'$. Then $(X,D)$ is log canonical on some open neighbourhood of $D'$ if and only if $(D',D''|_{D'})$ is log canonical.
\end{theorem}

\begin{definition} We say a variety $X$ is \emph{log canonical} if $(X,0)$ is log canonical. Note in particular that log canonical varieties are normal by assumption. \end{definition}

For a not necessarily log canonical pair $(X,D)$ we can still use the idea of log canonicity to measure singularities.

\begin{definition} Let $X$ be a normal variety, and let $D$ be a $\Q$-divisor. The \emph{log canonical threshold} of a pair $(X,D)$ is \begin{equation} \lct(X,D) = \sup \{\lambda\in\Q_{>0}\ | \  (X,\lambda D) \mathrm{\ is \ log \ canonical} \}. \end{equation}\end{definition}

One can generalise the log canonical threshold of a divisor to general coherent ideal sheaves as follows.

\begin{definition} Let $I \subset \scO_X$ be a coherent ideal sheaf, and let $D$ be an effective $\Q$-divisor on $X$. We say that $\pi: Y\to X$ is a \emph{log resolution} of $I$ and $D$ if $Y$ is smooth and there is an effective divisor $F$ on $Y$ with $\pi^{-1}I = \scO_Y(-F)$ such that $F \cup E\cup\tilde{D}$ has simple normal crossing support, where $\tilde{D}$ is the proper transform of $D$. Let $\pi: Y\to X$ be such a log resolution and assume the pair $(X,D)$ is log canonical. For a real number $c\in\R$, we define the \emph{discrepancy} of $((X,D);cI)$ to be \begin{equation} a(E_i,((X,D);cI)) = a(E_i,(X,D)) - c\val_{E_i}(I).\end{equation} Here by $\val_{E_i}(I)$ we mean the valuation of the ideal $I$ on $E_i$, while the $a(E_i,(X,D))$ are as in Definition \ref{lc}. We say $((X,D);cI)$ is \emph{log canonical} if $a(E_i,((X,D);cI))\geq -1$ for all $E_i$ appearing in a log resolution of $I$ and $D$. The \emph{log canonical threshold} of $((X,D);I)$ is then defined as \begin{equation} \lct((X,D);I) = \sup \{\lambda \in\Q_{>0}\ | \  ((X,D);\lambda I) \mathrm{\ is \ log \ canonical} \}. \end{equation} 
\end{definition} 

\begin{remark}\label{discrepineq}\cite[Proposition 8.5]{JK} For a proper birational map $f: X'\to X$, we have that \begin{equation} \lct((X,D);cI) \leq \min_{E_i\subset X'} \left\{\frac{1+\val_{E_i}K_{X'/X}-\val_{E_i}D}{c\val_{E_i}(I)}\right\},\end{equation} where our convention for the appearance of the $E_i$ is as in Definition \ref{lc}. Equality is achieved on a log resolution, where this is essentially a rephrasing of the definition of the log canonical threshold.\end{remark}

\begin{definition}\label{alphadef} Let $(X,L)$ be a log canonical polarised variety. We define the \emph{alpha invariant} of $(X,L)$ to be \begin{equation} \alpha(X,L) = \inf_{m \in \Z_{>0}}\inf_{D \in |mL|} \lct(X,\frac{1}{m}D).\end{equation} In particular, for $c>0$ the alpha invariant satisfies the scaling property \begin{equation}\alpha(X,cL) =\frac{1}{c} \alpha(X,L).\end{equation} \end{definition}

This definition of the alpha invariant is the algebraic counterpart of Tian's original definition. For further details on the following \emph{analytic} definition, see \cite[Appendix A]{CDS}.

\begin{definition} Let $h$ be a singular hermitian metric on $L$, written locally as $h=e^{-2\phi}$. We define the \emph{complex singularity exponent} $c(h)$ to be \begin{equation} c(h) = \sup \{\lambda\in\R_{>0} | \mathrm{\ for \ all \ } z\in X, h^{\lambda} = e^{-2\lambda\phi} \mathrm{ \ is\ } L^1 \mathrm{\ in \ a \ neighbourhood \ of \ }z \}.\end{equation} We then define Tian's \emph{alpha invariant} $\alpha^{an}(X,L)$ of $(X,L)$ to be \begin{equation} \alpha^{an}(X,L) = \inf_{h \mathrm{ \ with \ } \Theta_{L,h} \geq 0} c(h)\end{equation} where the infimum is over all singular hermitian metrics $h$ with curvature $\Theta_{L,h} \geq 0$. For a compact subgroup $G$ of $\Aut(X,L)$, one defines $\alpha^{an}$ similarly, however considering only $G$-invariant metrics.\end{definition}

\begin{theorem}\cite[Appendix A]{CDS} The alpha invariant $\alpha(X,L)$ defined algebraically equals Tian's alpha invariant $\alpha^{an}(X,L)$. That is, \begin{equation}\alpha(X,L) = \alpha^{an}(X,L).\end{equation}\end{theorem}

\begin{remark} As every divisor $D \in L$ gives rise to a singular hermitian metric, one sees that $\alpha(X,L) \geq \alpha^{an}(X,L)$. Equality follows from approximation techniques for plurisubharmonic functions. \end{remark}

The main consequence of the definition of the alpha invariant is the following theorem of Tian, which states that certain lower bounds on the alpha invariant imply the existence of a K\"{a}hler-Einstein metric. 

\begin{theorem}\cite[ Theorem 2.1]{GT}\label{tianalphatheorem} Let $G$ be a compact subgroup of $\Aut(X)$ and suppose $X$ is a smooth Fano variety with $\alpha_G(X,-K_X) > \frac{n}{n+1}$. Then $X$ admits a K\"{a}hler-Einstein metric. \end{theorem}

\section{Alpha Invariants and K-stability}

In this section we provide a \emph{sufficient} condition for polarised varieties $(X,L)$ of dimension $n$ to be K-stable. A fundamental quantity will be the \emph{slope} of a polarised variety.

\begin{definition} We define the \emph{slope} of $(X,L)$ to be \begin{equation} \mu(X,L) = \frac{-K_X.L^{n-1}}{L^n}=\frac{\int_Xc_1(X).c_1(L)^{n-1}}{\int_Xc_1(L)^n}.\end{equation} The slope of a polarised variety is thus a topological quantity which, after rescaling $L$, may be assumed equal to $1$. Note in particular that $\mu(X,-K_X)=1$.\end{definition}

\begin{remark} In \cite{RT}, Ross-Thomas defined a similar quantity, which they also call the slope, defined to be $\frac{n}{2}$ times our definition.\end{remark}   

\begin{theorem}\label{maintheorem} Let $(X,L)$ be a polarised $\Q$-Gorenstein log canonical variety with canonical divisor $K_X$. Suppose that 
\begin{itemize} 
\item[(i)] $\alpha(X,L)>\frac{n}{n+1}\mu(X,L)$ and
\item[(ii)] $-K_X \geq \frac{n}{n+1}\mu(X,L) L$.
\end{itemize}
Then $(X,L)$ is K-stable.\end{theorem}

\begin{remark}\label{lctorder}  Here, for divisors $H,H'$, we write $H\geq H'$ to mean $H-H'$ is nef. Note that both conditions are independent of positively scaling $L$. For $L=-K_X$, the second condition is vacuous, so in this case, this is the algebraic counterpart of a theorem of Tian (Theorem \ref{tianalphatheorem}) and in this case is due to Odaka-Sano \cite[Theorem 1.4]{OS}. The condition that $X$ is log canonical ensures that $\alpha(X,L)\geq 0$, while the condition that $X$ is $\Q$-Gorenstein ensures that $-K_X$ exists as a $\Q$-Cartier divisor. \end{remark}

\begin{remark}\label{calabiyau} If $\mu(X,L)=0$, i.e. $L^{n-1}.K_X=0$, the second condition requires $-K_X$ to be nef. Suppose $-K_X$ is nef but not numerically equivalent to zero, and suppose $L^{n-1}.K_X=0$. Then, by the Hodge Index Theorem \cite[Theorem 1]{TL} we would have $L^{n-2}.K_X^2<0$, contradicting the fact that $-K_X$ is nef. In particular, for the second condition of the theorem to hold, $X$ must either be numerically Calabi-Yau or Fano. In the Calabi-Yau case, this theorem also follows from a theorem due to Odaka \cite[Theorem 1.1]{O2}. \end{remark}

A Corollary of Theorem \ref{maintheorem} is that the automorphism group of $(X,L)$ is discrete. Indeed, if $\Aut(X,L)$ were to admit a one parameter subgroup, this would give two test configurations with Donaldson-Futaki invariants of opposite sign. But K-stability requires strict positivity of the Donaldson-Futaki invariant, a contradiction. 

\begin{corollary}\label{automorphisms} If the criteria of Theorem \ref{maintheorem} are satisfied, then $\Aut(X,L)$ is discrete. \end{corollary}

To prove Theorem \ref{maintheorem}, we first establish an \emph{upper} bound on the alpha invariant.  

\begin{proposition}\label{alphaprop}(c.f. \cite[Proposition 3.1]{OS}) Let $\scB$ be the blow-up of $X\times\pr^1$ along a flag ideal, with $\scB$ normal and Gorenstein in codimension one, $\scL-E$ relatively semi-ample over $\pr^1$ and notation as in Remark \ref{flagconventions}. Denote the natural map arising from the composition of the blow-up map and the projection map by $\Pi: \scB\to\pr^1$. Denote also

\begin{itemize}
\item the discrepancies as: $K_{\mathcal{B} / X \times \pr^1} = \sum a_iE_i$,
\item the multiplicities of $X\times \{0\}$ as: $\Pi^*(X\times \{0\}) = \Pi_*^{-1}(X\times \{0\}) + \sum b_i E_i$,
\item the exceptional divisor as: $\Pi^{-1}\mathcal{I} = \mathcal{O}_{\mathcal{B}}(-\sum c_i E_i)= \mathcal{O}_{\mathcal{B}}(-E).$\end{itemize} 

Then \begin{equation} \alpha(X,L) \leq \min_i\left\{\frac{a_i-b_i+1}{c_i}\right\}. \end{equation}\end{proposition}

\begin{proof} We are seeking an upper bound on the alpha invariant, where this upper bound is related to the flag ideal $\scI = I_0 + (t)I_1 + \hdots + (t^N)$ on $X\times\pr^1$. As the divisors considered in the definition of the alpha invariant are divisors on $X$, we pass from $\scI$ to its first component $I_0$. The choice of $I_0$ is because $I_0$ is a subsheaf of the full flag ideal $\scI$.

Let $\pi_0: Bl_{I_0}X \to X$ be the blow-up of $I_0$ with exceptional divisor $E_0$. We claim $\pi_0^*L-E_0$ is semi-ample. This is equivalent to $I_0^m(mL)$ being base-point free for some $m$. However, as $\scL-E$ is semi-ample restricted each fibre, we know that $\scI^m\pi^*(mL)$ is base-point free on each fibre of $X\times\pr^1$. As $I_0$ is a subsheaf of $\scI$, the result follows.

Choose $m$ sufficiently large and divisible such that $H^0(Bl_{I_0}X, m(\pi_0^*L - E_0)) = H^0(X,I_0^{m}(mL))$ has a section, which exists since multiples of semi-ample line bundles have sections. Let $D$ be in the linear series $H^0(X,I_0^{m}(mL))$. We show that \begin{equation}\alpha(X,L) \leq m\lct(X,D) \leq \min_i\left\{ \frac{a_i-b_i+1}{c_i}\right\}.\end{equation} 

For general ideal sheaves $I,J$, \cite[Property 1.12]{MM} states that $I \subset J$ implies $\lct(X,I) \leq \lct(X,J)$. We therefore see that \begin{equation} \lct(X,D) = \lct(X,I_D)  \leq  \frac{1}{m}\lct(X, I_0).\end{equation}

Note that $X\times\{0\}$ is a divisor on $X\times\A^1$. By a basic form of inversion of adjunction of log canonicity, we have \begin{equation} \lct(X, I_0)  =  \lct((X\times\pr^1, X\times\{0\});I_0).\end{equation} One can see this by taking a log resolution of $((X\times\pr^1,X\times\{ 0\});I_0)$ of the form $\tilde{X}\times\pr^1 \to X\times\pr^1$, where $\tilde{X} \to X$ is a log resolution of $(X,I_0)$. Note that for all divisors $E_i$ over $X$, we have \begin{equation}\val_{E_i}(I_0) \geq \val_{E_i}(\scI).\end{equation} In particular, we see that \begin{align} \lct((X\times\pr^1, X\times\{0\}); I_0) & \leq  \lct((X\times\pr^1, X\times\{0\});\scI)  \\ 
& \leq \min_i \left\{\frac{a_i-b_i+1}{c_i}\right\}. \end{align} The last inequality is by Remark \ref{discrepineq}.

\end{proof}

The final ingredients of the proof of Theorem \ref{maintheorem} are the following lemmas on computing the positivity of terms in Odaka's formula for the Donaldson-Futaki invariant, which are due to Odaka-Sano. We repeat their proof for the reader's convenience.

\begin{lemma}\label{nefdf}\cite[Lemma 4.2]{OS} Let $L$ and $R$ be ample and nef divisors respectively on $X$, with $p^*L = \scL$ and $p^*R = \scR$ where $p: \scB\to X$ is the natural map arising from the composition of the blow-up map $\scB\to X\times\pr^1$ and the projection $X\times\pr^1 \to \pr^1$. Suppose that $\scL-E$ is semi-ample on the blow-up $Bl_{\scI}X\times\pr^1$ for some flag ideal $\scI$. Then \begin{equation} (\scL-E)^n.\scR \leq 0.\end{equation} \end{lemma}

\begin{proof} 
Firstly note that, because $\scR$ and $\scL$ are the pull back of ample and nef divisors respectively from $X$, which has dimension $n$, we have $\scL^n.\scR=0$. Now note that we have the equality \begin{align} -(\scL-E)^n.\scR &= \scL^{n}.\scR - (\scL-E)^n.(\scR-E) - (\scL-E)^n.E \\ &= E.\scR.(\scL^{n-1}+\scL^{n-2}.(\scL-E)+\hdots +(\scL-E)^{n-1}).\end{align} As $\scL$ is nef and the restriction of $\scL-E$ to the central fibre of the map $\scB\to\pr^1$ is semi-ample, hence nef, and as $E.\scR$ is a non-zero effective cycle with support in the central fibre, the result follows. \end{proof}

\begin{lemma}\cite[Lemma 4.7]{OS}\label{excdf} Let $E=\sum c_iE_i$ be the exceptional divisor of the blow-up $\scB\to X\times\pr^1$. Then \begin{equation} (\scL-E)^n.E_i\geq0.\end{equation} Moreover, strict positivity holds for some $E_i$, that is, \begin{equation}(\scL-E)^n.E > 0.\end{equation} \end{lemma}

\begin{proof} Since each $E_i$ has support in the central fibre, and $\scL-E$ restricted to the central fibre is semi-ample, hence nef, we have that $(\scL-E)^n.E_i\geq0$. 

To show $(\scL-E)^n.E>0$, we first show $(\scL-E)^n.(\scL+nE) > 0$. Note that we have the equality of polynomials \begin{equation} (x-y)^n(x+ny) = x^{n+1} - \sum_{i=1}^n(n+1-i)(x-y)^{n-i}x^{i-1}y^2.\end{equation}

In fact, the polynomials $(x-y)^{n-i}x^{i-1}y^2$ for $1\leq i \leq n$ are linearly independent over $\Q$, and for all $0<s<n$, the monomial $x^sy^{n+1-s}$ can be written as a linear combination of these polynomials with coefficients in $\Z$.

Note that $\scL^{n+1}=0$, as $\scL$ is the pull back of an ample line bundle from $X$, which has dimension $n$. In particular, we can write 

\begin{equation}\label{techeq1}  (\scL-E)^n.(\scL+nE) = -E^2.\left(\sum_{i=1}^n(n+1-i)(\scL-E)^{n-i}.\scL^{i-1}\right). \end{equation}

Let $s=\dim(\Supp(\scO / \scI))$, where $\scI = I_0 + (t)I_1 +\hdots +(t^N)$. By dividing $\scI$ by a power of $t$ if necessary, which does not change the resulting blow-up $Bl_{\scI}X\times\pr^1$ and hence does not change the Donaldson-Futaki invariant of the associated semi-test configuration, we can assume $s<n$. Perturbing the coefficients in equation (\ref{techeq1}), we get 

\begin{equation}(\scL-E)^n.(\scL+nE) = -E^2.\left(\sum_{i=1}^n(n+1-i+\epsilon_i)(\scL-E)^{n-i}\scL^{i-1}\right) - \epsilon'(\scL^s.(-E)^{n+1-s}) \end{equation} where $0<|\epsilon_i|\ll 1$ and $0 < \epsilon' \ll 1$.

The following lemma then shows that $(\scL-E)^n.(\scL+nE) > 0$.

\begin{lemma}\cite[Lemma 4.7]{OS} \begin{enumerate}[(i)] 
\item $-E^2.(\scL-E)^{n-i}.\scL^{i-1} \geq 0$ for all $0<i<n$.
\item $\scL^s.(-E)^{n+1-s} < 0$.
\end{enumerate} \end{lemma}

\begin{proof} $(i)$ Cutting $\scB$ by general elements of $|r\scL|$ and $|r(\scL-E)|$ for $r\gg0$, we can assume $\dim X = 2$. In this case, the required equation becomes $-E^2.(\scL-E)\geq 0$. Note that $\scL-E$ is semi-ample restricted to fibres of $\scB\to\pr^1$ and $E$ has support in the central fibre. In particular, $E.(\scL-E)$ is an effective cycle with support in fibres of the blow-up map $\scB\to X\times\pr^1$. Since $-E$ is relatively ample over fibres of $\scB\to X\times\pr^1$, we have $-E.(E.(\scL-E))\geq 0$ and the result follows.

$(ii)$ Again cutting $\scB$ by general elements of $\scL^r$ for $r\gg0$, we can assume $s=0$. The required result then follows by relative ampleness of $-E$ over fibres of $\scB\to X\times\pr^1$.

\end{proof}

Finally, since $(\scL-E)^n.(\scL+nE)>0$ and $(\scL-E)^n.\scL \leq 0$ by Lemma \ref{nefdf}, we have $(\scL-E)^n.E>0$ as required.

\end{proof}

\begin{proof} (of Theorem \ref{maintheorem}) We show that the Donaldson-Futaki invariant is positive for all semi-test configurations of the form $(\scB,\scL^r-E)$ arising from flag ideals. We assume $r=1$ for notational simplicity, the proof in the general case is essentially the same. The idea is to first split formula \ref{dfformulapf} for the Donaldson-Futaki invariant into two terms, which we consider separately. We split the Donaldson-Futaki invariant as \begin{align} &\DF(\scB, \scL - E) = \DF_{num}+\DF_{disc}, \\
&\DF_{num} = (\scL-E)^n.(-n(L^{n-1}.K_X)\scL + (n+1)(L^n)\scK_X), \\ 
&\DF_{disc} = (\scL-E)^n.((n+1)(L^n)K_{\scB / X \times \pr^1} + n(L^{n-1}.K_X)E). \end{align}

Our second hypothesis in Theorem \ref{maintheorem} is \begin{equation} -K_X \geq \frac{n}{n+1}\mu(X,L)L.\end{equation} In particular, $n(L^{n-1}.K_X)\scL - (n+1)(L^n)\scK_X$ is nef. So, by Lemma \ref{nefdf}, $\DF_{num}\geq 0$.

As $(\scL-E)^n.E > 0$ by Lemma \ref{excdf}, it suffices to show that there exists an $\epsilon>0$ such that \begin{equation}(n+1)(L^n)K_{\scB / X \times \pr^1} + n(L^{n-1}.K_X)E \geq \epsilon E.\end{equation} Here we mean that each coefficient of $E_i$ is non-negative in the difference of the divisors. As $L^n$ is positive, this is equivalent to showing \begin{equation}K_{\mathcal{B} / X \times \pr^1} - \frac{n}{n+1}\mu(X,L)E \geq \epsilon E.\end{equation} By the first assumption in \ref{maintheorem}, namely that $\alpha(X,L) > \frac{n}{n+1}\mu(X,L)$, we see that \begin{equation}K_{\mathcal{B} / X \times \pr^1} - \frac{n}{n+1}\mu(X,L)E > K_{\mathcal{B} / X \times \pr^1} - \alpha(X,L)E. \end{equation} But by the upper bound on the alpha invariant, Proposition \ref{alphaprop}, we see that
\begin{equation} K_{\mathcal{B} / X \times \pr^1} - \alpha(X,L)E \geq K_{\mathcal{B} / X \times \pr^1} - \min_i\{\frac{a_i-b_i+1}{c_i}\}E. \end{equation} Here we have used notation as in Proposition \ref{alphaprop}. Finally, we see that \[ \begin{array}{lcl}
K_{\mathcal{B} / X \times \pr^1} - \frac{n}{n+1}\mu(X,L)E & > &K_{\mathcal{B} / X \times \pr^1} - \min_i\{\frac{a_i-b_i+1}{c_i}\}\sum c_iE_i \\
& = &\sum a_iE_i - \min_i\{\frac{a_i-b_i+1}{c_i}\}\sum c_iE_i \\
& = &\sum (\frac{a_i-b_i+1}{c_i} - \min_i\{\frac{a_i-b_i+1}{c_i}\} + \frac{b_i-1}{c_i})c_iE_i \\
& \geq & 0.
\end{array} \] The result follows as $(\scL-E)^n.E > 0$, by Lemma \ref{excdf}.

\end{proof}

\begin{remark} One can marginally strengthen Theorem \ref{maintheorem} as follows. The positivity of the alpha invariant is used in the proof of Theorem \ref{maintheorem} as it appears as a coefficient of the exceptional divisor $E$. In particular, one has a term with positive contribution of the form $(\alpha(X,L) - \frac{n}{n+1}\mu(X,L))E$. By the proof of Lemma \ref{excdf}, we have that $(\scL-E)^n.(\scL+nE)>0$. Using this, one can use the positivity of the contribution of the term $(\alpha(X,L) - \frac{n}{n+1}\mu(X,L))E$ to slightly weaken the requirement  that $-K_X \geq \frac{n}{n+1}\mu(X,L) L$. However, the resulting hypothesis still implies that $-K_X$ is either ample or numerically trivial. We therefore omit the details. \end{remark}

\begin{remark} For any compact subgroup $G\subset \Aut(X,L)$, Odaka-Sano \cite[Section 2.2]{OS} have defined a form of stability, which they call \emph{G-equivariant K-stability} and conjecture to be equivalent to K-stability. \end{remark}

\begin{definition} Let $G \subset \Aut(X,L)$ be compact, and define a $G$-test configuration $(\scX,\scL)$ be a test configuration equipped with an extension of the natural $G$-action on $(\scX,\scL)|_{\pi^{-1}(\A^1 - \{0\})}$ to $(\scX,\scL)$ which commutes with the $\C^*$-action on $(\scX,\scL)$. We say $(X,L)$ is \emph{G-equivariantly K-stable} if the Donaldson-Futaki invariant of all $G$-test configuration is strictly positive for all non-trivial test configurations. \end{definition}

As $G$-test configurations give rise to $G$-invariant flag ideals, Theorem \ref{maintheorem} can be adapted to $G$-equivariant K-stability as follows.

\begin{corollary}  Let $(X,L)$ be a polarised $\Q$-Gorenstein log canonical variety with canonical divisor $K_X$, and let $G\subset \Aut(X,L)$ be a compact subgroup. Suppose that
\begin{itemize} 
\item[(i)] $\alpha_G(X,L)>\frac{n}{n+1}\mu(X,L)$ and
\item[(ii)] $-K_X \geq \frac{n}{n+1}\mu(X,L) L$.
\end{itemize}
Then $(X,L)$ is $G$-equivariantly K-stable.\end{corollary}

\section{Examples}

By showing that the alpha invariant of a line bundle is a continuous function of the line bundle, we first show that the conditions of Theorem \ref{maintheorem} are \emph{open} when varying the polarisation. To prove this, we need a lemma regarding adding ample divisors and alpha invariants.

\begin{lemma}\label{lctampleineq} Let $(X,L)$ be a log canonical polarised variety, and let $D$ be an ample $\Q$-divisor on $X$. Then $\alpha(X,L+D)\leq\alpha(X,L)$. \end{lemma}

\begin{proof} Take any divisor $D'\in |m'L|$. We find a divisor $F \in|p(L+D)|$ such that $\lct(X,\frac{1}{p}F) \leq \lct(X,\frac{1}{m'}D')$. Suppose that $mD$ is a $\Z$-divisor. Let $F=mD'+mm'D \in |mm'(L+D)|$. As $F-mD'=mm'D$ is ample, hence effective, the discrepancies satisfy \begin{equation} a(E_i,(X,F)) \leq a(E_i,(X,mD'))\end{equation} for all divisors $E_i$ over $X$ \cite[Lemma 2.27]{KM}, so we have \begin{equation} \lct(X,\frac{1}{mm'}F) \leq \lct(X,\frac{1}{m'}D').\end{equation} Therefore $\alpha(X,L+D) \leq \alpha(X,L)$. \end{proof}

Using this we can show that the alpha invariant of a polarised variety $(X,L)$ is a continuous function of $L$.
 
\begin{proposition}\label{contlct}Let $(X,L)$ be a polarised klt variety and $D$ be a $\Q$-divisor on $X$. Then for all $\epsilon>0$ there exists a $\delta>0$ depending on $D$ such that $|\alpha(X,L) - \alpha(X,L+\delta D)| < \epsilon$.\end{proposition}

\begin{proof} Firstly, suppose both $\gamma L+D$ and $\gamma L-D$ are ample for some $0 < \gamma < 1$. By the inverse linearity property of the alpha invariant noted in Definition \ref{alphadef}, we then have \begin{equation}\alpha(X,L) = (1+\gamma)\alpha(X,(1+\gamma)L).\end{equation} Lemma \ref{lctampleineq} implies that subtracting ample divisors raises the alpha invariant. Applying Lemma \ref{lctampleineq} by subtracting $\gamma L-D$ from $(1+\gamma)L$, we see that \begin{equation} (1+\gamma)\alpha(X,(1+\gamma)L) \leq (1+\gamma)\alpha(X,L+D).\end{equation} This in particular implies \begin{equation}\alpha(X,L) - \alpha(X,L+D) \leq \gamma\alpha(X,L+D).\end{equation} On the other hand, since $\gamma L+D$ is ample, applying Lemma \ref{lctampleineq} by adding $\gamma L+D$ to $(1-\gamma)L$ gives \begin{align}\alpha(X,L) &= (1-\gamma)\alpha(X,(1-\gamma)L) \\ & \geq (1-\gamma)\alpha(X,L+D). \label{ineqtech} \end{align} Therefore \begin{equation}|\alpha(X,L) - \alpha(X,L+D)| \leq \gamma\alpha(X,L+D).\end{equation} Note that equation (\ref{ineqtech}) implies that $\alpha(X,L+D) \leq \frac{1}{1-\gamma}\alpha(X,L)$, so we have \begin{equation}\label{conteqn}|\alpha(X,L) - \alpha(X,L+D)| \leq \frac{\gamma}{1-\gamma}\alpha(X,L).\end{equation}

Since ampleness is an open condition, there exists a $c>0$ such that both $L+cD$ and $L-cD$ are ample. 

We now show continuity of the alpha invariant at $L$. Given $\epsilon>0$ let $\delta = \frac{c\epsilon}{2\alpha(X,L)+\epsilon}$. Then both $(\delta c^{-1})L+ \delta D$ and $(\delta c^{-1})L-\delta D$ are ample. In our situation $\gamma = \delta c^{-1} = \frac{\epsilon}{2\alpha(X,L)+\epsilon} < 1$, so we can apply equation (\ref{conteqn}). Noting that $\frac{\delta c^{-1}}{1-\delta c^{-1}} = \frac{\epsilon}{2\alpha(X,L)}$, we therefore have \begin{equation}|\alpha(X,L) - \alpha(X,L+\delta D)| \leq \frac{\epsilon}{2}<\epsilon.\end{equation}

\end{proof}

\begin{corollary}\label{open} Suppose $(X,L)$ is a klt $\Q$-Gorenstein polarised variety such that
\begin{itemize} 
\item[(i)] $\alpha(X,L)>\frac{n}{n+1}\mu(X,L)$ and
\item[(ii)] $-K_X > \frac{n}{n+1}\mu(X,L) L$.
\end{itemize}
Note that both inequalities are strict. Then for all $\Q$-divisors $D$, there exists an $\epsilon>0$ such that $L + \epsilon D$ is K-stable. \end{corollary}

\begin{proof} This follows by Proposition \ref{contlct} and continuity of intersections of divisors. \end{proof}

We now apply Theorem \ref{maintheorem} to a general degree one del Pezzo surface $X$. Here the genericity condition means that $|-K_X|$ contains no cuspidal curves, so $\alpha(X,-K_X)=1$ (\cite{IC}, Theorem 1.7). Note that $X$ is isomorphic to the blow-up of $\pr^2$ at a configuration of $8$ points in general position. Denote by $H$ the hyperplane divisor pulled back from $\pr^2$, let $E_i$ be the $8$ exceptional divisors arising from an isomorphism $X\cong Bl_{p_1,...,p_8}\pr^2$ and let $L_{\lambda} = 3H - \sum_{i=1}^7E_i - \lambda E_8$. 

\begin{theorem}\label{delpezzo} $(X,L_{\lambda})$ is K-stable for \begin{equation}\frac{19}{25}\approx \frac{1}{9}(10-\sqrt{10}) < \lambda < \sqrt{10}-2\approx \frac{29}{25}.\end{equation} \end{theorem}

To prove this result, we first obtain a lower bound for $\alpha(X,L_{\lambda})$ using the following two lemmas.

\begin{lemma}\label{convexity}\cite[Corollary 6]{AW} Let $S$ be a smooth variety, let $p\in S$, and let $D,B$ be effective $\Q$-divisors on $S$ with $p\in S, p\in B$ such that $(S,D)$ is not log canonical at $p$ but $(S,B)$ is log canonical at $p$. Then, for all $c\in [0,1)\cap \Q$, \begin{equation} (S,\frac{1}{1-c}(D-cB))\end{equation} is not log canonical at $p$. \end{lemma}

\begin{lemma}\label{multiplicity}\cite[Lemma 2.4 (i)]{JMG} Let $(S,D)$ be pair consitisting of a smooth surface $S$ and an effective $\Q$-divisor $D$ such that $(S,D)$ is not log canonical at $p$. Then $\mult_pD > 1$.\end{lemma}

\begin{proposition}\label{alphacalc} For $X$ be a general del Pezzo surface of degree one as above, and $L_{\lambda}=3H-\sum_{i=1}^7 E_i - \lambda E_8$ with $\lambda \geq 0$, we have\begin{equation} \alpha(X,L_{\lambda})\geq \min\left\{\frac{1}{2-\lambda}, 1\right\}.\end{equation} \end{proposition}

\begin{proof} Suppose for contradiction $\omega <  \min\{\frac{1}{2-\lambda}, 1\}$, and there exists an effective $\Q$-divisor $D$ with $mD\in |mL_{\lambda}|$ for some $m$ such that $(S,\omega D)$ is not log canonical at some point $p\in X$. Write $D = aC + \Omega$, where $C \in |-K_X|$ is a $\Z$-divisor with $p\in C$, and $C\not\subseteq\Supp(\Omega)$. Note that since $X$ is a \emph{general} degree one del Pezzo surface, we have that $\omega C$ is log canonical by (\cite{JP}, Proposition 3.2). Since $\Omega = D -aC$, we see that \begin{align} \Omega.H &= (D-aC).H  \\ &=(1-a)(-K_X).H + (1-\lambda)E_8.H \\ &= 3(1-a). \end{align} But since $H$ is ample and $\Omega$ is effective, $\Omega.H\geq 0$.  Thus $a\leq 1$, and in particular, $\omega a < 1$.

By Lemma \ref{convexity}, we see that $(S, \frac{1}{1-\omega a}(\omega D - \omega a C))$ is not log canonical at $p$. Note that $\omega D - \omega a C = \omega \Omega$. Therefore $\mult_p(\frac{\omega}{1-\omega a}\Omega)>1$ by Lemma \ref{multiplicity}. But since $C\not\subseteq\Supp(\Omega)$, we have that \begin{equation}\omega C.\Omega \geq \omega \mult_p\Omega> 1-\omega a.\end{equation} Thus \begin{align} \omega(2-\lambda) &= \omega D.C  \\ &=\omega(aC.C + \Omega.C) \\ &> \omega a + 1 - \omega a \\ &= 1. \end{align} But this implies $\omega > \frac{1}{2-\lambda}$, a contradiction.

\end{proof}

Using this lower bound we can prove Theorem \ref{delpezzo}.

\begin{proof} (of Theorem \ref{delpezzo}) For Theorem \ref{maintheorem} to apply, the two equations that must be satisfied are 

\begin{itemize}
\item[(i)] $\alpha(X,L_{\lambda}) > \frac{2}{3}\mu(X,L_{\lambda})$ and
\item[(ii)] $-K_X - \frac{2}{3}\mu(X,L_{\lambda})L_{\lambda}$ is nef.
\end{itemize}

In our case $\mu(X,L_{\lambda}) = \frac{2-\lambda}{2-\lambda^2}$. By Proposition \ref{alphacalc}, for $\lambda\leq 1$, we have $\alpha(X,L_{\lambda}) \geq \frac{1}{2-\lambda}$ and the first condition is always satisfied. When $\lambda\geq 1$, we have $\alpha(X,L_{\lambda}) \geq 1$ and the first condition requires $2-3\lambda^2+2\lambda > 0$, which is true for $\lambda < \frac{1}{3}(1+\sqrt{7})$.

For the second condition to apply, we require \begin{equation} 3H - \sum_{i=1}^7E_i- \frac{6-4\lambda-\lambda^2}{2+2\lambda - 3\lambda^2}E_8\end{equation} to be nef. Note for $\lambda=1$ this holds. 

By the cone theorem \cite[Theorem 3.7]{KM} applied to a del Pezzo surface, to check when a line bundle on a del Pezzo surface is nef, it suffices to check it has non-negative intersection with all curves of negative self-intersection. However, by the adjunction formula, all curves $C$ on a del Pezzo surface of negative self-intersection are exceptional, that is, $C.C=-1$. Therefore, to check when a line bundle is nef on a del Pezzo surface, it suffices to check it has non-negative intersection with all exceptional curves. For the blow-up of $\pr^2$ at $8$ points, from \cite[Theorem 26.2]{YM} we know that the exceptional curves are the proper transforms of:

\begin{itemize}
\item points which are blown up, with class $E_i$, \label{cond:1}
\item lines through pairs of points, with class $H-E_i-E_j$, \label{cond:2}
\item conics through $5$ points, with class $2H-\sum_5E_i$, \label{cond:3}
\item cubics through $7$ points, vanishing doubly at $E_j$ for some $j$, with class $3H - E_j - \sum_7E_i$, \label{cond:4}
\item quartics through $8$ points, vanishing doubly at $E_j, E_k, E_l$, with class $4H-E_i- E_j- E_k-\sum_8E_l$, \label{cond:5}
\item quintics through $8$ points, vanishing doubly at $6$ points, with class $5H - E_j - E_k- 2\sum_6E_i$, \label{cond:6}
\item sextics through $8$ points, vanishing doubly at $7$ points and triply at another, with class $6H - 3E_j - 2\sum_7E_i $. \label{cond:7}
\end{itemize}

For a line bundle of the form $W= 3H - \sum^{7}_{i=1}E_i - \delta E_8$, the first condition requires $\delta\geq 0$, the second and third conditions require $\delta\leq 2$, the fourth, fifth and sixth require $\delta \leq \frac{3}{2}$, while the seventh condition requires $\delta \leq \frac{4}{3}$. In particular, $\delta=\frac{4}{3}$ is the maximal value of $\delta$ with $W$ nef.

The equation that therefore must be satisfied for $3H - \sum_{i=1}^7 E_i - \frac{6-4\lambda-\lambda^2}{2+2\lambda - 3\lambda^2}E_8$ to be nef is \begin{equation}0\leq \frac{6-4\lambda-\lambda^2}{2+2\lambda - 3\lambda^2} \leq \frac{4}{3}.\end{equation} As $\lambda>0$, the condition that $6-4\lambda-\lambda^2\geq 0$ requires $\lambda<\sqrt{10}-2\approx \frac{29}{25}$. The upper bound is equivalent to \begin{equation} 9\lambda^2 - 20\lambda + 10 \leq 0, \end{equation} which is true for $\frac{1}{9}(10-\sqrt{10}) \leq \lambda \leq \frac{1}{9}(10+\sqrt{10})\approx \frac{29}{20}$. Therefore, the range for which both conditions required to apply Theorem \ref{maintheorem} are satisfied is $\frac{1}{9}(10-\sqrt{10}) < \lambda < \sqrt{10}-2$.

\end{proof}

\begin{remark} Note that the lower bound for $\alpha(X,L_{\lambda})$ may not be sharp. A more delicate analysis of the $\Q$-divisors linearly equivalent to $L_{\lambda}$ may provide a sharper lower bound. However, both the upper and lower bounds obtained in Theorem \ref{delpezzo} were given by the requirement that $-K_X \geq \frac{2}{3}\mu(X,L_{\lambda})L_{\lambda}$. Since we calculated exactly for which $\lambda$ that this condition holds, we have calculated precisely the range of $\lambda$ for which Theorem \ref{maintheorem} applies. \end{remark}

\begin{remark}Analytically, Arezzo-Pacard \cite[Theorem 1.1]{AP} have shown that if a general $(X,L)$ admits a constant scalar curvature K\"{a}hler metric in $c_1(L)$, and $\pi: Y\to X$ is the blow-up of $X$ at a point $p$, then $(Y,\pi^*L-\epsilon E)$ admits a constant scalar curvature K\"{a}hler metric in $c_1(\pi^*L-\epsilon E)$ for sufficiently small $\epsilon$, provided $\Aut(X,L)$ is discrete. As the existence of a cscK metric in $c_1(L)$ implies K-stability by work of Stoppa \cite[Theorem 1.2]{JS1} and Donaldson \cite{SD}, this in particular implies that $(X,L_{\lambda})$ as in Theorem \ref{delpezzo} is K-stable for $\lambda$ sufficiently small. Theorem \ref{delpezzo} shows that $(X,L_{\lambda})$ is K-stable for $\frac{1}{9}(10-\sqrt{10}) < \lambda < \sqrt{10}-2$. On the other hand, using the techniques of slope stability, Ross-Thomas \cite[Example 5.30]{RT2} have shown that there are polarisations of a general degree one del Pezzo surface $X$ which are K-\emph{un}stable. It would be interesting to know exactly which polarisations of a general degree one del Pezzo surface are K-stable. \end{remark}

\section{Log Alpha Invariants and Log K-stability}

In this section we extend Theorem \ref{maintheorem} to K-stability with cone singularities along an anti-canonical divisor, which conjecturally corresponds to the existence of cscK metrics with cone singularities along a divisor. For a general introduction to log K-stability, see \cite{OS2}.

\begin{definition} Let $(X,L^r)$ be a normal polarised variety, and let $D$ be an effective integral reduced divisor on $X$. We define a \emph{log test configuration} for $((X,D);L^r)$ to be a pair of test configurations $(\scX,\scL)$ for $(X,L^r)$ and $(\scY,\scL|_{\scY})$ for $(D,L^r|_D)$ with a compatible $\C^*$ action. We denote by $((\scX,\scY);\scL)$ the data of a log test configuration. Denote the Hilbert polynomials of $(X,L)$ and $(D,L|_D)$ respectively as \begin{align}& \scP(k) = \chi(X,L^k)=a_0k^n+a_1k^{n-1}+O(k^{n-2}), \\ & \tilde{\scP}(k) = \chi(D,L|_D^k)=\tilde{a}_0k^{n-1}+\tilde{a}_1k^{n-2}+O(k^{n-3}).\end{align}Denote also the total weights of the $\C^*$-actions on $H^0(\scX_0,\scL^k_0)$ and $H^0(\scY_0,\scL^{k_0}|_{\scY_0})$ respectively as \begin{align}&w(k) = b_0k^{n}+b_1k^{n-1}+O(k^{n-2}), \\ &\tilde{w}(k) = \tilde{b}_0k^{n+1}+\tilde{b}_1k^n+O(k^{n-1}).\end{align} We define the \emph{log Donaldson-Futaki invariant} of $((\scX,\scY); \scL)$ with cone angle $2\pi\beta$ for $0\leq\beta \leq 1$ to be \begin{equation} \DF_{\beta}((\scX,\scY);\scL)= 2(b_0a_1 - b_1 a_0) + (1-\beta)(a_0\tilde{b}_0-b_0\tilde{a}_0).\end{equation} We say that $((X,D);L)$ is \emph{log K-stable with cone angle} $2\pi\beta$ if $\DF_{\beta}((\scX,\scY);\scL)>0$ for all log test configurations $((\scX,\scY);\scL)$ with $\scX,\scY$ normal, Gorenstein in codimension one and such that $((\scX,\scY);\scL)$ is not almost trivial. \end{definition}

Note that the usual Donaldson-Futaki invariant for the test configuration $(\scX,\scL)$ is $\frac{b_0a_1 - b_1 a_0}{a_0^2}$, we have multiplied by $2a_0^2$ for ease of notation. Since K-stability is independent of positively scaling $L$, this makes no difference to the definition of K-stability.

Odaka-Sun \cite{OS2} have extended the blowing-up formalism of Odaka to the log case. Recall that to certain flag ideals $\scI$ on $X\times\A^1$ one can associate a semi-test configuration by the following method. Blowing-up $\scI$ on $X\times\pr^1$, we get a map \begin{equation} \pi: \scB=Bl_{\scI}(X\times\pr^1)\to X\times\pr^1.\end{equation} Denote $\scB_{(D\times\pr^1)}=Bl_{\scI|_{(D\times\pr^1)}}(D\times\pr^1)$ and let $E$ be the exceptional divisor of the blow-up $\pi: \scB\to X\times\pr^1$, that is, $\scO(-E)=\pi^{-1}\scI$. Abusing notation, write $\scL-E$ to denote $(p_1\circ \pi)^*L\otimes \scO(-E)$, where $p_1:X\times\pr^1\to X$ is the natural projection. 

\begin{theorem}\cite[Corollary 3.6]{OS2} A normal polarised variety $(X,L)$ is log K-stable with cone angle $2\pi\beta$ if and only if $\DF_{\beta}((\scB,\scB_{(D\times\pr^1)});\scL^r-E)>0$ for all $r>0$ and for all flag ideals $\scI$ such that $\scB, \scB_{(D\times\pr^1)}$ are normal and Gorenstein in codimension one, $\scL^r-E$ is relatively semi-ample over $\pr^1$ and $\scI \neq (t^N)$.\end{theorem}

Moreover, there is an explicit formula for the log Donaldson-Futaki invariant for log test configurations arising from flag ideals.

\begin{theorem}\cite[Theorem 3.7]{OS2} With all notation as above, we have \begin{align}\label{logdfform} \DF_{\beta}&((\scB,\scB_{(D\times\pr^1)});\scL-E) = -n(L^{n-1}.(K_X+(1-\beta)D))(\scL-E)^{n+1} + \\ &+ (n+1)(L^n)(\scL-E)^n.(\scK_X + (1-\beta)\scD + (K_{\scB / ((X,(1-\beta)D)\times\pr^1})_{exc}).\end{align}  Here we have denoted by $\scK_{\scB / ((X,(1-\beta)D)\times\pr^1})_{exc}$ the exceptional terms of $K_{\scB} - \pi^*(K_{X\times\pr^1}+(1-\beta)D)$, and $\scK_X$ the pull back of $K_X$ to $\scB$. The intersection numbers $L^{n-1}.K_X$ and $L^n$ are computed on $X$, while the remaining intersection numbers are computed on $\scB$. Replacing $L$ and $\scL$ by $L^r$ and $\scL^r$ respectively in formula \ref{logdfform} gives the formula for the Donaldson-Futaki invariant of a test configuration of the form $(\scB, \scL^r-E)$. \end{theorem}

We can extend the definition of the alpha invariant to the log setting as follows.

\begin{definition} Let $((X,D);L)$ consist of a log canonical pair $(X,D)$ with $L$ an ample $\Q$-line bundle. We define the log alpha invariant of $((X,D);L)$ to be \begin{equation} \alpha((X,D);L) = \inf_{m \in \Z_{>0}}\inf_{F \in |mL|} \lct((X,D);\frac{1}{m}F). \end{equation}\end{definition}

Berman \cite[Section 6]{RB} has provided an analytic counterpart to the log alpha invariant as follows.

\begin{definition}Let $((X,D);L)$ consist of a Kawamata log terminal pair $(X,D)$ with $X$ smooth, $L$ an ample $\Q$-line bundle and $D=\sum d_iD_i$ a simple normal crossing divisor with $D_i = \{f_i=0\}$. Let $h$ be a singular hermitian metric on $L$, written locally as $h=e^{-2\phi}$. We define the \emph{complex singularity exponent} $c_D(h)$ to be \begin{equation} c(h) = \sup \left\{\lambda\in\R_{>0} | \mathrm{\ for \ all \ } z\in X, h^{\lambda} = e^{-2\lambda\phi}\prod |f_i|^{-\lambda d_i}  \mathrm{ \ is\ } L^1 \mathrm{ \ near \ } z \right\}.\end{equation} We then define Tian's \emph{log alpha invariant} $\alpha^{an}((X,D);L)$ of $((X,D);L)$ to be \begin{equation} \alpha^{an}((X,D);L) = \inf_{h \mathrm{ \ with \ } \Theta_{L,h} \geq 0} c(h)\end{equation} where the infimum is over all singular hermitian metrics $h$ with curvature $\Theta_{L,h} \geq 0$. For a compact subgroup $G$ of $\Aut(X,L)$, one defines $\alpha^{an}$ similarly, however considering only $G$-invariant metrics.\end{definition}

\begin{theorem} \cite[Section 6]{RB} The log alpha invariant $\alpha((X,D);L)$ defined algebraically equals Tian's log alpha invariant $\alpha^{an}((X,D);L)$. That is, \begin{equation}\alpha((X,D);L) = \alpha^{an}((X,D);L)\end{equation}\end{theorem}

We can now extend Theorem \ref{maintheorem} to the log setting. 

\begin{theorem}\label{logmaintheorem} Let $((X,D);L)$ consist of a $\Q$-Gorenstein log canonical pair $(X,D)$ with canonical divisor $K_X$, such that $D$ is an effective integral reduced Cartier divisor on a polarised variety $(X,L)$. Denote $\mu_{\beta}((X,D);L)= \frac{-(K_X+(1-\beta)D).L^{n-1}}{L^n}$ Suppose that 
\begin{itemize}
\item[(i)] $\alpha((X,D);L)>\frac{n}{n+1}\mu_{\beta}((X,D);L)$ and
\item[(ii)] $-(K_X +(1-\beta)D)\geq \frac{n}{n+1}\mu_{\beta}((X,D);L) L$.
\end{itemize}
Then $((X,D);L)$ is log K-stable with cone angle $\beta$ along $D$.\end{theorem}

\begin{remark} In the case $L=-K_X$, and $D\in|-K_X|$, this result is due to Odaka-Sun \cite[Theorem 5.6]{OS2}. Again in the case $L=-K_X$, this is the analytic counterpart of a result of Berman (\cite{RB}, Theorem 3.11) and Jeffres-Mazzeo-Rubinstein (\cite{JMR}, Lemma 6.9). Explicit examples are given in \cite{CR}.\end{remark}

To prove this theorem, we extend Proposition \ref{alphaprop} to the log setting.

\begin{proposition}\label{logalphaprop} Let $\scB$ be the blow-up of $X\times\pr^1$ along a flag ideal, with $\scB$ normal and Gorenstein in codimension one, $\scL-E$ relatively semi-ample over $\pr^1$ and notation as in Remark \ref{flagconventions}. Denote the natural map arising from the composition of the blow-up map and the projection map by $\Pi: \scB\to\pr^1$. Denote also

\begin{itemize}
\item the discrepancies as: $K_{\mathcal{B} / X \times \pr^1} = \sum a_iE_i$,
\item the multiplicities of $X\times \{0\}$ as: $\Pi^*(X\times \{0\}) = \Pi_*^{-1}(X\times \{0\}) + \sum b_i E_i$,
\item the multiplicities of $D$ as: $\Pi^*D = \Pi_*^{-1}D + \sum d_i E_i$,
\item the exceptional divisor as: $\Pi^{-1}\mathcal{I} = \mathcal{O}_{\mathcal{B}}(-\sum c_i E_i)= \mathcal{O}_{\mathcal{B}}(-E).$
\end{itemize} 

Then \begin{equation} \alpha((X,(1-\beta)D);L) \leq \min_i\left\{\frac{a_i-b_i+1-(1-\beta)d_i}{c_i}\right\}. \end{equation}\end{proposition}

\begin{proof} Let $\pi_0: Bl_{I_0}X \to X$ be the blow-up of $I_0$ with exceptional divisor $E_0$. As in Proposition \ref{alphaprop}, we have that $\pi_0^*L-E_0$ is semi-ample. Choose $m$ sufficiently large and divisible such that $H^0(Bl_{I_0}X, m(\pi_0^*L - E_0)) = H^0(X,I_0^{m}(mL))$ has a section, and let $F$ be such a section. We show that \begin{equation}\alpha((X,(1-\beta)D);L) \leq m\lct((X,D);F) \leq \min_i\left\{ \frac{a_i-b_i+1-(1-\beta)d_i}{c_i}\right\}.\end{equation} 

Since for general ideal sheaves $I,J$, we have $I \subset J$ implies $\lct(X,I) \leq \lct(X,J)$, we see that \begin{equation} \lct((X,D);F) = \lct((X,D);I_F)  \leq  \frac{1}{m}\lct((X,D);I_0).\end{equation}

Since $(X,D)$ is log canonical, using inversion of adjunction of log canonicity (Theorem \ref{inversionofadjunction}), we have 
\begin{align} \lct((X,(1-\beta)D);I_0)  & =  \lct((X\times\pr^1, X\times\{0\}+(1-\beta)D\times\pr^1);I_0) \\ & \leq  \lct((X\times\pr^1, X\times\{0\} + (1-\beta)D\times\pr^1 );\scI)  \\ & \leq \min_i \left\{\frac{a_i-b_i+1-(1-\beta)d_i}{c_i}\right\}. \end{align}

The last inequality is by Remark \ref{discrepineq}.

\end{proof}

Using this we can prove Theorem \ref{logmaintheorem}.

\begin{proof} (of Theorem \ref{logmaintheorem}) We treat the case $r=1$ for notational simplicity, the general case is similar. The log Donaldson-Futaki invariant is given by 
\begin{align} \DF_{\beta}&((\scB,\scB_{(D\times\pr^1)});\scL-E) = -n(L^{n-1}.(K_X+(1-\beta)D))(\scL-E)^{n+1} + 
\\ &+ (n+1)(L^n)(\scL-E)^n.(\scK_X + (1-\beta)\scD + (K_{\scB / ((X,(1-\beta)D)\times\pr^1})_{exc}).\end{align} For ease of notation, we let $K_X' = K_X+(1-\beta)D$ and $\scK' = \scK_X + (1-\beta)\scD$. We split the log Donaldson-Futaki invariant into two terms as 
\begin{align} &\DF_{\beta}((\scB,\scB_{(D\times\pr^1)});\scL-E) = \DF_{\beta,num}+\DF_{\beta,disc}, \\
&\DF_{\beta,num} = (\scL-E)^n.(-n(L^{n-1}.K_X')\scL + (n+1)(L^n)\scK_X'), \\ 
&\DF_{\beta,disc} = (\scL-E)^n.((n+1)(L^n)(K_{\scB / ((X,(1-\beta)D)\times\pr^1})_{exc}) + n(L^{n-1}.K_X') E). \end{align}
Since $-(K_X +(1-\beta)D)\geq \frac{n}{n+1}\mu_{\beta}((X,D);L) L$, Lemma \ref{nefdf} implies that $\DF_{\beta,num}\geq 0$.

To prove $\DF_{\beta,disc} > 0$, since $(\scL-E)^n.E>0$ by Lemma \ref{excdf} it suffices to prove the existence of an $\epsilon>0$ such that \begin{equation} \scK_{\scB / ((X,(1-\beta)D)\times\pr^1})_{exc} - \frac{n}{n+1}\mu_{\beta}((X,D);L) E \geq \epsilon E.\end{equation} By Proposition \ref{logalphaprop} and the first hypothesis of the theorem, we have that 

\begin{align}
& (K_{\scB / ((X,(1-\beta)D)\times\pr^1})_{exc} - \frac{n}{n+1}\mu_{\beta}((X,D);L)E > \\
& (K_{\scB / ((X,(1-\beta)D)\times\pr^1})_{exc} - \alpha((X,D);L)E \\
& \geq \sum (a_i-(1-\beta)d_i)E_i - \min_i\left\{\frac{a_i-b_i+1-(1-\beta)d_i}{c_i}\right\}\sum c_iE_i \\
& = \sum \left(\frac{a_i-b_i-(1-\beta)d_i+1}{c_i} - \min_i\left\{\frac{a_i-b_i-(1-\beta)d_i}{c_i}+1\right\} + \frac{b_i-1}{c_i}\right)c_iE_i \\
& \geq  0.
\end{align} The result follows.\end{proof}

\noindent {\bf Acknowledgements:}  I would like to thank my supervisor Julius Ross for his support, advice and for many useful discussions. I would also like to thank Jesus Martinez-Garcia, Ivan Cheltsov, Costya Shramov and John Ottem for helpful conversations. Finally I would like to thank the referee for useful comments, in particular for pointing out Corollary \ref{automorphisms}.

\vspace{4mm}

\noindent Ruadha\'i Dervan, University of Cambridge,  UK \\ R.Dervan@dpmms.cam.ac.uk
\end{document}